\documentclass{amsart}

\usepackage{amsthm} 
\usepackage{color}

\usepackage{graphicx,epsfig}
\usepackage{epstopdf}
\usepackage{amsmath, amssymb, latexsym, euscript,amscd}
\usepackage{url}
\usepackage[all]{xy}
\usepackage{psfrag}
\usepackage{mathrsfs}

\definecolor{refkey}{rgb}{1,0,1} 
\definecolor{labelkey}{rgb}{1,0,1}

\usepackage{graphicx,epsfig}
\usepackage{amsmath, amssymb, latexsym, euscript}

\usepackage{psfrag}

\usepackage{url}
\usepackage[all]{xy}
\usepackage{psfrag}
\usepackage{hyperref}

\definecolor{darkgreen}{rgb}{0.0, 0.5, 0.0}
\definecolor{purple}{rgb}{0.5, 0.0, 0.5}

\newtheorem{theorem}{Theorem}[section]
\newtheorem{lemma}[theorem]{Lemma}
 
\newtheorem{definition}[theorem]{Definition} 
\newtheorem{proposition}[theorem]{Proposition}

\def\gap{\vspace{.3cm}\noindent}

\def\smallskip{\vspace{.15cm}}
\def\medskip{\vspace{.3cm}}
\def\text{\mbox}
\def\rh2{{\mathbb R}{\mathbb H}^2}
\def\ch2{{\mathbb C}{\mathbb H}^2}
\def\RP{\operatorname{\mathbb RP}}

\def\RPn{{\mathbb RP}^n}
\def\RP{{\mathbb RP}}

\def\R{\mathbb R}
\def\Hom{\operatorname{Hom}}
\def\PGL{\operatorname{PGL}}

\DeclareMathOperator{\area}{area}

\begin{document}

\title{The Area of Convex Projective Surfaces and Fock-Goncharov Coordinates}

\author{Ilesanmi Adeboye}
\author{Daryl Cooper}
\subjclass[2010]{57M50 Convex projective surface. Shape Parameter. Volume.}
\begin{abstract}
The area of a convex projective surface of genus $g\ge 2$ is at least $(g-1)\pi^2/2+\|\tau\|^2/8$ where
$\tau=(\log t_i)$ is the vector of triangle invariants of Bonahon-Dreyer and $t_i$ are the
Fock-Goncharov triangle coordinates.
\end{abstract}

\maketitle

A {\em convex projective surface} is $F=\Omega/\Gamma$ where $\Omega\subset\RP^2$ is 
the interior of a compact convex set disjoint from some projective line,
and $\Gamma\subset\PGL(3,\R)$ is a discrete subgroup that preserves $\Omega$
and acts freely on it. An example is a hyperbolic surface.
 Let $\RP(S)$ be the space of marked convex real-projective structures on a closed orientable surface $S$.
Goldman and Choi \cite{CG} showed that this can by identified with the 
Hitchin component of $\Hom(\pi_1S,\PGL(3,\R))/\PGL(3,\R)$. Bonahon and Dreyer \cite{BONAHON} showed:

\begin{theorem} Suppose $S$ is a closed orientable surface of genus $g\ge 2$. 
There are real analytic maps,
  the {\em triangle invariant} $\tau:\RP(S)\to\R^{4g-4}$, and the {\em shear invariant} $\sigma:\RP(S)\to \R^{12g-10}$, 
  so that $\beta=(\tau,\sigma):\RP(S)\to\R^{16g-14}$ is a real-analytic parameterization. The image of $\beta$
  is the an open cone  defined by 2 linear inequalities and 2 linear equalities. The image
  of $\tau$ is open in a codimension-1 subspace.   \end{theorem}

The Hilbert metric $d_{\Omega}$ on a bounded convex set  $\Omega\subset {\mathbb R}^2$ is a Finsler metric
normalized so that when $\Omega$ is the interior of  the unit disc the Hilbert metric is the
 hyperbolic
 metric (curvature $-1$). A Finsler metric determines a
measure   called {\em p-area} or just {\em area} that is
 the largest measure in the same measure class as Lebesgue
so that 
the area of an infinitesimal parallelogram is at most the product of the side lengths, 
see  (\ref{meas:def}).
In general this is {\em not Hausdorff measure} \cite{PAIVA} but coincides with it for Riemannian metrics. 
For example the unit ball in the
sup-norm is a square of side-length 2 which has p-area equal to the Euclidean
area $4$, and in  the taxicab norm the unit ball has p-area 2 but the  Hausdorff measure of the unit ball of a norm
is always $\pi$.
When applied to the Hilbert metric we call p-area  the {\em Hilbert area} or just {\em area}.
 In the positive quadrant the Hilbert area form is  $dxdy/(4xy)$ see (\ref{positivequad}).

An ideal triangle in $\Omega$ has a {\em shape parameter} $t>0$. This is one of the coordinates introduced by Fock
 and Goncharov \cite{FG}. The {\em triangle invariant} in the Bonahon-Dreyer theorem 
 for this triangle is $\tau=\log t$ 
which is a certain {\em signed distance}, see
 Figures (\ref{inscribetriangle})  and (\ref{domaintriangle}).
 Our main result is:
\begin{theorem}\label{maintheorem} Suppose $\mathcal T$ is an ideal triangulation of a surface $F$.
The map $\alpha_{\mathcal T}:\RP(S)\to\R$ defined by $$\alpha_{\mathcal T}:(F)=(g-1)\pi^2/2+\|\tau\|^2/8$$
satisfies $\area(F)\ge\alpha(F)$. Here $\tau=(\tau_1,\cdots,\tau_{4g-4})$ are the components of $\tau$ and $\|\tau\|^2=\sum\tau_i^2$.
\end{theorem}

\noindent  This follows from:

\begin{proposition}\label{triangleprop} If $T$ is an ideal triangle with shape parameter $t$ in a properly convex domain $\Omega$ then $$area_{\Omega}(T)\ge(\pi^2+(\log t)^2)/8$$with equality iff $\Omega$ is the interior of a triangle.
\end{proposition}


 The area of a 
compact hyperbolic surface $F$ is $2\pi|\chi (F)|$ and in particular is bounded below. 
A corollary of the above gives a {\em lower bound} on the area of a convex projective surface
 by using an ideal triangulation of the surface. 
\begin{theorem}\label{thm} If $F$ is a compact, properly convex projective surface then
 $\area(F)\ge (\pi/2)^2\cdot |\chi (F)|$.\end{theorem}

It follows that if $Q$ is a compact $2$-orbifold with $\chi^{orb}(Q)<0$ 
then every convex projective structure on $Q$ has area at least $(\pi/2)^2\cdot |\chi^{orb}(Q)|$.
In particular this area is bounded below by $\pi^2/168$. The first author  has given lower bounds
on the volumes of {\em hyperbolic orbifolds} \cite{A1}; and with {\em Guofang Wei} see also \cite{A2} and 
for {\em complex hyperbolic orbifolds} \cite{A3}.

With the above  normalization the Hilbert metric on the interior of the unit disc equals the hyperbolic metric. 
The above lower bound  is $\pi/8\approx 38\%$ of the hyperbolic area. In the remainder
of this paper we normalize the Hilbert metric so that it is {\em twice} that above. This means the areas calculated in
the rest of the paper should be divided by $4$ to give the results announced.

This result extends to complete convex projective structures $2$-orbifolds and to surfaces with cusps. Theorem \ref{maintheorem} suggests
various questions. For example there is a graph $\Gamma$ with a vertex for each ideal triangulation $\mathcal T$ 
of $F$ and edges that correspond to edge flips. Given a strictly convex structure on $F$ 
what are the properties
of the function defined on the vertices of $\Gamma$ by $\mathcal T\mapsto \alpha_{\mathcal T}(F)$ ?
Is the maximum attained ? Is there a uniform upper bound on the difference between the maximum and the area of $F$ ?

 It follows from the Margulis lemma for
  convex projective orbifolds \cite{CLT1}, \cite{CM1} that there is a lower bound on the Hilbert volume of a strictly convex projective 
  $n$-orbifold. What are explicit lower bounds ?
  
 The proposition is proved by using the following observation. An ideal triangle in a domain
 $\Omega$ meets $\partial\Omega$ at its three vertices. There is a triangle $\Delta$ that contains 
 $\Omega$ and is
 tangent to $\Omega$ at these three points, see Figure (\ref{inscribetriangle}). The Hilbert metric given by $\Delta$ is
 smaller than that given by $\Omega$. Thus  the area of 
 $T$ in the Hilbert metric on $\Omega$ is bigger than its area
 using the Hilbert metric from $\Delta$. We explicitly calculate the latter. The theorem follows
 by using ideal triangulations.
 
 After the first version of this paper was written the authors learned of \cite{COLBOIS}. In that paper the authors show that
 the (Hausdorff) area of an ideal triangle is bounded below, and also that it is not bounded above.
 Their lower bound is different because they
 use a different definition of {\em area}.
 However, as Marquis remarks in \cite{AROUND}, although there are different notions of area for Finsler metrics,
 because of the Benz{\'e}cri compactness theorem \cite{benz}, there is a universal bound on the ratios
 of different area forms for reasonable choices. The point of this paper is the relation to Fock-Goncharov coordinates.

The authors thank the MRC at Snowbird 2011 and the MSRI during spring of 2015 for hospitality while parts of this paper
were written. The authors acknowledge support from U.S. National Science Foundation grants DMS 1107452, 1107263, 1107367 ``RNMS: GEometric structures And Representation varieties'' (the GEAR Network). 
Cooper was partially supported by NSF grants 1065939, 1207068 and  1045292 and
 thanks IAS and the Ellentuck Fund for partial support during completion of this paper. 

\
\section{Length and Area in Hilbert geometry}

The \textbf{cross-ratio} of four distinct points $y_1,y_2,y_3,y_4\in\mathbb{R}$ is 
\[
cr(y_1,y_2,y_3,y_4)=\frac{(y_1-y_3)}{(y_1-y_2)}\frac{(y_2-y_4)}{(y_3-y_4)}
\]
Using the embedding $\R\subset\RP^1$ given by $y/x\mapsto[x:y]$ cross-ratio extends to a continuous map
$cr:X\longrightarrow\RP^1$ where $X\subset\left(\RP^1\right)^4$ is the subset
of quadruples of points at least 3 of which are distinct
$$cr([x_1:y_1],[x_2:y_2],[x_3:y_3],[x_4:y_4])=[(x_2y_1-x_1y_2)(x_4y_3-x_3y_4):(x_3 y_1-x_1y_3 ) (x_4 y_2-x_2y_4 )]$$

Suppose $\Omega\subset\R^n\subset\RPn$ is an open convex set that contains no affine line.
Given $b,c\in \Omega$ there is   a projective line $\ell$ in $\RPn$ that contains them. This line meets
 $\partial\overline\Omega\subset \RPn$ in two distinct points
$a,d$. Label these points so  $a,b,c,d$ are in linear order
along  $\ell\cap\R^n$. Since $\Omega$ contains no affine line $a\ne d$. 
The {\em Hilbert metric on $\Omega$} is$$d_{\Omega}(b,c)= | \log cr(a,b,c,d) |$$
Some authors use $(1/2)$ of this so that when $\Omega$ is the unit ball then $d_{\Omega}$ 
has curvature $-1$. This is a Finsler metric and
$$d_{\Omega}(x,x+dx)=\left(\frac{1}{|x-a|}+\frac{1}{|x-b|}\right)dx$$
In particular if $\Omega=(0,\infty)\subset\R^1$ then $a=0,b=\infty$ and $d_{\Omega}(x,x+dx)=dx/x$.

The literature contains many distinct definitions of area for Finsler metrics \cite{PAIVA},\cite{area2}.  
These depend on how {\em area} is defined for a normed vector space $(V,\|\cdot\|)$. Area is a Borel measure on $V$ 
that is preserved by translation (ie a Haar measure) so
it is some multiple of Lebesgue measure.
We will adopt the following definition which is particularly well suited for our purposes.

\begin{definition}\label{meas:def} Suppose $(V,\|\cdot\|)$ is a normed 2-dimensional real vector space. 
Choose an inner product on $V$ and let $\lambda$ be the resulting Lebesgue measure. Set
$K=\sup\lambda(\{\alpha a+\beta b\ :\ 0\le \alpha,\beta\le 1\ \})$ where the supremum is over all
$a,b\in V$ with $\|a\|,\|b\|\le 1$. Then define 
the {\em parallelogram measure} or {\em p-area} $\omega_{\|\cdot\|}$ on $V$ by
 $\omega_{\|\cdot\|}=K^{-1}\cdot\lambda$.
\end{definition}
It is easy to check the definition is independent of the choice of inner product. 
 Having done this
we dispense with the inner product, and refer to the norm of a vector as its {\em length}.
The definition is equivalent to declaring that
the maximum area of a parallelogram with sides of unit length  is $1$.  If the norm
is given by an inner product then parallelogram measure coincides with the usual
area. This definition generalizes to $n$ dimensions using parallelopipeds spanned by
$n$ vectors of norm one.

A {\em rectangle} in $(V,\|\cdot\|)$ is defined to be any parallelogram with side lengths $x$ and $y$
and area $xy$. Such parallelograms always exist. This enables the standard construction
of Lebesque measure in the plane, starting from an inner product, to be extended to an arbitrary norm on the plane.

Parallelogram measure is an {\em increasing} function of the metric in the sense that
if $\|\cdot \|$ and $\|\cdot\|'$ are two norms on $V$ with $\|\cdot\|\le\|\cdot\|'$ then $\omega_{\|\cdot\|}\le\omega_{\|\cdot\|'}$.
In particular if $\alpha>0$ then $\omega_{\alpha\|\cdot\|}=\alpha^2\omega_{\|\cdot\|}$.

A {\em Finsler surface} is a pair $(S,ds)$ where $S$ is a smooth surface and $ds_x$  is a norm on $T_xS$ for each $x\in S$. 
The {\em p-area form} on $S$ is the p-area form
for $ds_x$ on $T_xS$. For a properly convex projective surface $S$ the resulting area form $\omega_S$
 is called the {\em Hilbert area form}, and
 the {\em Hilbert area of $S$} is $\int_s \omega_S$.
If $\Omega'\subset\Omega$ are properly convex
then $d_{\Omega}\le d_{\Omega'}$ on $\Omega'$ and $\mu_{\Omega}\le\mu_{\Omega'}$.

\section{Hex geometry}

A reference for this section is \cite{delaharpe}. 
Let $u_0,u_1,u_2\in {\mathbb R}^2$ be unit vectors with respect to the standard inner product
such that $u_0+u_1+u_2=0$. We will use $u_0=(1,0)$ and $u_1=(-1/2,\sqrt{3}/2)$ then $u_2=-u_0-u_1$. The convex hull
of the vectors $\{\pm u_0,\pm u_1,\pm u_2\}$ is a regular Hexagon $H$.
\begin{definition}The {\em Hex plane} $({\mathcal H},d_{\mathcal H})$ is the metric space obtained from the  norm on 
${\mathbb R}^2$ with unit ball $H$. 
\end{definition}

For the Hex plane p-area is different to 
Busemann volume, used for example in \cite{AROUND}, or Holmes-Thompson used in \cite{CV}. 
On a normed plane all these measures are multiples of Haar measure, and so they are multiples of each other.
We will
describe some properties of p-area for the Hex plane which suggest this is the {\em right} definition to use.
\begin{lemma}\label{conversionfactor}  Let $\omega_{\mathcal H}$ denote the area form on the Hex plane and $\lambda$ be Lebesgue measure with respect to the standard inner product. Then
$\omega_{\mathcal H}=(2/\sqrt{3})\lambda$  
\end{lemma}
\begin{proof} Let $a,b$ be unit vectors in the Hex norm. Then they lie on the boundary of
the regular unit Hexagon $H$
 center at the origin
and determine a parallelogram $P(a,b)$ with vertices $\{0,a,b,a+b\}$.
Suppose $a$ lies on the edge $e$ of $H$. Then the area of $P(a,b)$ is maximized by taking
 $b$ to be a vertex of $H$ that that is not an endpoint of $\pm e$. The Euclidean area of $P(a,b)$ is then $\sqrt{3}/2$.
 \end{proof}
 
  Our choice of normalization of area has the following consequences in the Hex plane.
A Hex circle of radius $r$ is a Euclidean regular Hexagon
so that the (Euclidean=Hex)  distance from the center to a vertex is $r$. The circumference of this circle is $6r$ and
the p-area  is $3r^2$.

The {\em positive quadrant} is $Q=\{(x,y)\ :\ x,y>0\ \}$. A {\em triangle} in $\RP^2$ is a compact convex subset, $\Delta$, 
bounded by 3 segments of projective lines.
There is a projective transformation
taking the interior of $\Delta$ to $\{[x:y:1]\ :\ x,y>0\ \}$ which may be identified with $Q$. Thus the Hilbert metric
on the interior of $\Delta$ is isometric to $(Q,d_Q)$.

 \begin{lemma}[Proposition 7 in \cite{delaharpe}]\label{isom} The Hilbert metric on $Q$ is isometric to the Hex plane $({\mathcal H},d_{\mathcal H})$.
\end{lemma}
\begin{proof} There is an isometry 
$A:\mathbb (\R^2,\|\cdot\|_{Hex})\rightarrow\mathbb (Q,d_Q)$ given by
 \[A(u,v)=\left(e^{u+\frac{1}{\sqrt{3}}v},e^{\frac{2}{\sqrt{3}}v}\right)=(x,y).\]
This may be checked as follows. The map $A$ conjugates the action of $\R^2$ on itself by
translations to the action of the positive diagonal group on $Q$. Thus it suffices to check
$A$ is infinitesimally an isometry at the origin. \end{proof}

\begin{figure}
	\psfrag{P}{$\color{blue}P$}
	\psfrag{H}{$H$}
	\psfrag{Z}{$v_0+v_1$}
	\psfrag{Y}{$v_0+u_1$}
	\psfrag{W}{$v_1$}
	\psfrag{0}{$0$}
	\psfrag{B}{\color{blue} base $=1$}
	\psfrag{C}{\color{darkgreen} height $=1$}
	\psfrag{L}{\color{purple} base $=1$}
	\psfrag{Q}{$\color{darkgreen}P'$}
	\psfrag{N}{$u_1$}
	\psfrag{S}{$u_0+u_1$}
	\psfrag{T}{$\color{purple}P''$}
	\psfrag{M}{$u_0$}
	\psfrag{V}{$v_0$}
	\includegraphics[scale=.8]{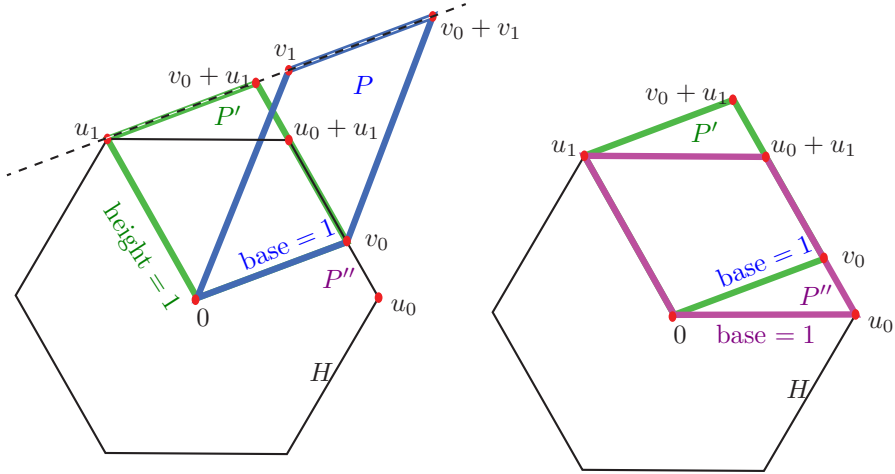}
	\caption{In Hex geometry area = base $\times$ height}\label{parallelogramarea}
\end{figure}

If $\Omega\subset\RP^2$ is a properly convex domain the {\em Hilbert area form} $\omega_{\Omega}$ is the 2-form
given by the   Hilbert metric. 

\begin{lemma}\label{positivequad} The Hilbert area form on the positive quadrant $Q=\{(x,y)\ :\ x,y>0\ \}$ is
$$\omega_{Q}=\frac{dxdy}{xy}$$
\end{lemma}
\begin{proof}  The  isometry $A$ in (\ref{isom})  infinitesimally multiplies Lebesgue measure $\lambda$ on $\R^2$ by
$$
J=\begin{vmatrix}
\frac{\partial x}{\partial u} & \frac{\partial x}{\partial v} \\ 
\frac{\partial y}{\partial u} & \frac{\partial y}{\partial v}
\end{vmatrix} =\begin{vmatrix}
e^{u+\frac{1}{\sqrt{3}}v} & \frac{1}{\sqrt{3}}e^{u+\frac{1}{\sqrt{3}}v} \\ 
0 & \frac{2}{\sqrt{3}}e^{\frac{2}{\sqrt{3}}v}
\end{vmatrix} =\frac{2}{\sqrt{3}}\left(e^{u+\frac{1}{\sqrt{3}}v}\right)\left(e^{\frac{2}{\sqrt{3}}v}\right)
=\frac{2}{\sqrt{3}}xy.
$$
Since $\omega_{\mathcal H}=(2/\sqrt{3})\lambda$  it follows that $$\omega_Q=\frac{\omega_{\mathcal H}}{\frac{2}{\sqrt{3}}xy}=\frac{\lambda}{xy}=\frac{dxdy}{xy}$$ \end{proof}

\begin{lemma}\label{Hexparallelogram} The p-area of every parallelogram  in the Hex plane
is base $\times$ height.
\end{lemma}
 \begin{proof}   Affine maps of the plane multiply Euclidean area (and hence p-area) by the determinant of the linear part.
 Hence it suffices to prove the result in the special case of a parallelogram $P$ with vertices $0,v_0,v_1,v_0+v_1$ 
 for which the length of the base is $\|v_0\|_{Hex}=1$  and the height is also one. Moreover a rotation through an angle of $\pi/3$
 is an isometry of Hex. Thus we may assume $v_0$
 is  on the side of $H$ between $u_0$ and $u_0+u_1$.
 
 Refer to Figure (\ref{parallelogramarea}). We transform $P$ to $P'$ to $P''$ and show that each of these parallelograms
 has the same p-area.
 A shear parallel to the base of $P$ preserves base, height and p-area. 
 Shear $P$ parallel to $v_0$ sending $v_1$ to $u_1$ to give a parallelogram $P'$ with vertices
 $0,v_0,u_1,v_0+u_1$. Now shear $P'$ parallel to the $u_1$-direction to get a parallelogram  $P''$  with vertices $0,u_0,u_1,u_0+u_1$.
 This shear preserves base and height because of the special properties of $H$. 
 The area of $P''$ is $1$.
   \end{proof}


\section{Ideal triangles in Hex geometry}

Suppose $\Omega\subset\RP^2$ is an open properly convex set. 
 If $T$ is a triangle with vertices in $\partial \Omega$ then $T\cap\Omega$ is called
an {\em ideal triangle} in $\Omega$. It is {\em proper} if $T\cap\partial\Omega$ consists
of only the vertices of $T$.  We will only be concerned with proper ideal triangles, and will
henceforth omit the term {\em proper}.

If $\Delta=\{[x_0v_0+x_1v_1+x_2v_2]\ :\ x_i>0\ \}$ is the interior of
 triangle in $\RP^2$ there is an isometry $\phi:(\Delta,d_{\Delta})\longrightarrow ({\mathcal H},d_{\mathcal H})$
given by
$$\phi[x_0v_0+x_1v_1+x_2v_2]=(\log x_0) u_0+(\log x_1)u_1+(\log x_2)u_2$$
Suppose $T$ is an ideal triangle in $\Delta$. We refer to $\phi(T)$ as an {\em ideal triangle} in the Hex plane.
Then $T\subset\Delta$ together with an ordering of the vertices of $\Delta$ determines a {\em shape parameter} $t=t(T,\Delta)\in{\mathbb R}$ defined as follows, see Figure (\ref{inscribetriangle}). If the vertices of $T$ are $[w_0=v_0+av_1],[w_1=v_1+bv_2],[w_2=v_2+cv_0]$ with $a,b,c>0$ then $t=abc$.
This depends only on the cyclic
ordering of the vertices. Changing this ordering replaces $t$ by $1/t$. Observe that when $a=b=1$ then $|\log t|$ is the Hilbert distance
in $(v_1,v_2)$ between $[w_2]$ and the midpoint $[v_1+v_2]$.

The group $PGL(\Omega)$ is the subgroup  of $PGL(3,{\mathbb R})$ which 
preserves $\Omega$. Using the basis $v_0,v_1,v_2$ of ${\mathbb R}^3$ the identity component  $PGL_0(\Delta)$ consists of positive diagonal matrices.
This group acts transitively on the interior of $\Delta$. There is a unique element $\tau\in PGL_0(\Delta)$ which takes  two of the vertices of $T$ to
$[u_1+u_2],[u_1+u_3]$. The remaining vertex is taken to $[u_2+t u_3]$, see Figure (\ref{inscribetriangle}). The {\em regular} ideal triangle is given by
$t=1$. It has maximal isometry group: dihedral of order $6$.
\begin{figure}
	\psfrag{a}{$[1:0:0]$}
	\psfrag{b}{$[0:1:0]$}
	\psfrag{c}{$[0:0:1]$}
	\psfrag{e}{$[1:a:0]$}
	\psfrag{f}{$[0:1:b]$}
	\psfrag{g}{$[c:0:1]$}
	\psfrag{h}{$t=abc$}
	\psfrag{k}{$\Delta$}
	\psfrag{l}{$[1:1:0]$}
	\psfrag{m}{$[1:0:1]$}
	\psfrag{n}{$[0:1:t]$}
	\psfrag{w}{$[0:1:1]$}
	\psfrag{v}{$\tau = \log t$}
		\psfrag{T}{$\tau$}
	\psfrag{G}{$\color{red}T$}
	\includegraphics[scale=0.6]{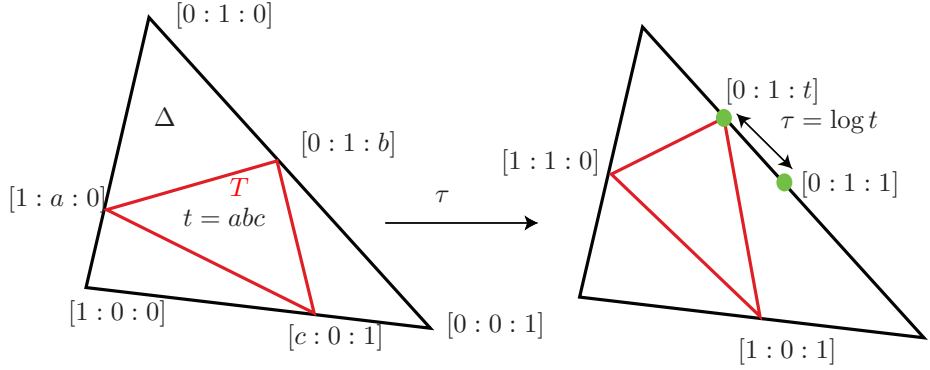}
	\caption{The shape parameter $t>0$ and triangle invariant $\tau=\log t$}\label{inscribetriangle}
\end{figure}

\begin{proposition}\label{shapeclass} Isometry classes of ideal Hex triangle  are 1-1 correspondence with shape parameters $t\in[1,\infty)$.
\end{proposition}

\begin{lemma} The Hilbert area of an ideal triangle in the Hex plane with shape parameter $t$ is
$$B(t)=\int_{1}^{\infty}\frac{1}{s}\log\left(\frac{(st+1)(s+t)}{t(s-1)^2}\right)\ ds$$
\end{lemma}
\begin{proof} By (\ref{isom}) the ideal triangle is isometric to an ideal triangle $T$ in $Q$. By means
of a projective transformation preserving $Q$  we can arrange that one side of $T$ is given by $x+y=1$ and the other two sides
are then the parallel rays in $Q$ given by $y-1=tx$ and $y=t(x-1)$.  Refer to Figure \ref{domain1}.
It is easy to check that $t$ is the shape parameter. 

For $s\ge 1$ define $\alpha(s),\beta(s)$ to be the points of intersection of the line $x+y=s$ with the sides
$y-1=tx$ and $y=tx+1$ of $T$. Let $\gamma(s)$ be the line segment with these endpoints. 
For $s\ge 1$ these lines
foliate $T$.
For $s>1$ define $$\ell(s)=d_Q(\alpha(s),\beta(s))$$ to be the Hilbert length of $\gamma(s)$.
The distance between $\gamma(s)$ and $\gamma(s+ds)$ is $ds/s$. This is easily seen by projecting onto the $x$-axis along the direction $x+y=0$.
It follows from (\ref{Hexparallelogram}) that the Hilbert area of the infinitesimal parallelogram shown is $\ell(s)ds/s$, thus
$$area_{\Delta}(T )=\int_1^{\infty}\ell(s)\frac{ds}{s}$$
Next we compute
$$\ell(s)=\log\left(\frac{(s+t)(st+1)}{t(s-1)^2}\right)$$
This may be done as follows. Projection of one line onto another preserves Hilbert distance.
Let $\pi$ be vertical projection of the line segment $x+y=s$ in $Q$ onto the $x$-axis. The image
of this segment is $[0,s]$.
Then $$\pi(\alpha(s))=\frac{s-1}{t+1},\qquad\pi(\beta(s))=\frac{s+t}{t+1}$$
Then $\ell(s)=d_Q(\alpha(s),\beta(s)) =d_{[0,s]}(\alpha(s),\beta(s))=|\log CR(0,\alpha(s),\beta(s),s)|$. 

\end{proof}

Next we calculate this integral. This implies Proposition \ref{triangleprop}.
\begin{lemma} $$B(t)=\frac{\pi^2+\left(\log t\right)^2}{2}$$
\end{lemma}

\begin{proof} 
$$\frac{dB}{dt}
=\frac{d}{dt}\left(\int_{1}^{\infty}\frac{1}{s}\log\left(\frac{(st+1)(s+t)}{t(s-1)^2}\right)\ ds\right)$$
Taking the derivative inside the integral gives
\begin{align*}
\int_{1}^{\infty}\frac{1}{s}\frac{d}{dt}\left(\log\left(\frac{(st+1)(s+t)}{t(s-1)^2}\right)\right)\ ds 
& =
\int_{1}^{\infty}\frac{1}{s}\left(\frac{s}{1+s t}+\frac{1}{s+t}-\frac{1}{t}\right)\ ds \\
& = \int_{1}^{\infty}\left(\frac{1}{1+ st}-\frac{1}{t(s+t)}\right)\ ds \\
& = \left[\frac{1}{t}\left(\log(1 + st)-\log (s+t)\right)\right]_{s=1}^{s=\infty} \\
& = \left[t^{-1}\log\left(\frac{1 + st}{s+t}\right)\right]_{s=1}^{s=\infty}\\
& = t^{-1}\log t
\end{align*}
Thus $$\frac{dB}{dt}=t^{-1}\log t$$
Integrating gives $B(t)=(1/2)(\log t)^2 + C$. The next lemma shows $C=\pi^2/2$.
\end{proof}

\begin{figure}
	\psfrag{a}{$y=t(x-1)$}
	\psfrag{b}{$y-1=tx$}
	\psfrag{c}{$x+y=1$} 
	\psfrag{e}{$1$}
	\psfrag{f}{$s$}
	\psfrag{g}{$s+ds$}
	\psfrag{d}{$ds/s$}
	\psfrag{u}{$\alpha(s)$}
	\psfrag{v}{$\beta(s)$}
	\psfrag{W}{$\ell(s)$}
	\psfrag{Q}{$\Delta=Q$}
	\psfrag{T}{$T$}
		\psfrag{r}{$t=1$}
	\psfrag{s}{$t=100$}
	\includegraphics[scale=0.7]{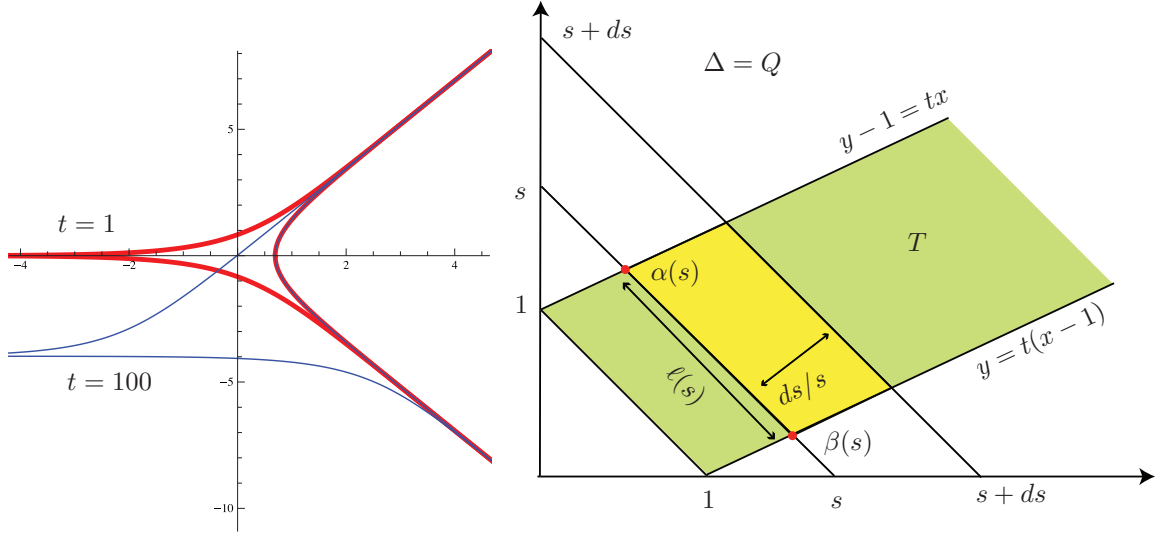}
	\caption{Ideal triangles in the Hex plane \qquad An ideal triangle in the positive quadrant.}\label{domain1}
\end{figure}

\begin{lemma} $$B(1) = \frac{\pi^2}{2}$$
\end{lemma}
\begin{proof} $$B(1)=\int_{1}^{\infty}\frac{1}{s}\log\left(\frac{(s+1)^2}{(s-1)^2}\right)\ ds 
= \int_{1}^{\infty}\frac{2}{s}\log\left(\frac{s+1}{s-1}\right)\ ds$$
Set $w=(s+1)/(s-1)$ then $$s=\frac{w+1}{w-1}=1+\frac{2}{w-1}\quad\text{and}\quad\frac{ds}{dw}=-2(w-1)^{-2}$$
hence $$B(1)=\int_{1}^{\infty}\frac{2}{s}\log\left(\frac{s+1}{s-1}\right)\ ds = \int_{\infty}^{1}2\left(\frac{w-1}{w+1}\right)\log w\ (-2)(w-1)^{-2} dw
=  \int_{1}^{\infty}\frac{4\log w \ dw}{w^2-1} $$
Set $w=e^x$ this becomes
$$B(1)=\int_0^{\infty}\frac{4x\ e^xdx}{e^{2x}-1}=\int_0^{\infty}\frac{4x\ dx}{e^{x}-e^{-x}}$$
The integrand is even so $$B(1)=\int_{-\infty}^{\infty}\frac{2x\ dx}{e^{x}-e^{-x}}$$
Define $$g(z)=\frac{i}{4\pi}z(z-2\pi i)$$ then $$g(z)-g(z+2\pi i)=z$$ 
Hence
$$B(1)=\int_{-\infty}^{\infty}\frac{2x\ dx}{e^{x}-e^{-x}} = 2\int_{-\infty}^{\infty}\frac{g(x)-g(x+2\pi i)}{e^{x}-e^{-x}}\ dx$$
\begin{figure}
	\psfrag{a}{$-R$}
	\psfrag{b}{$R$}
	\psfrag{c}{$R+2\pi i$}
	\psfrag{d}{$-R+2\pi i$}
	\psfrag{g}{$\Gamma_R$}
	\psfrag{e}{$\pi i$}
	\psfrag{f}{Residue = $-1/2$}
	\includegraphics[scale=0.98]{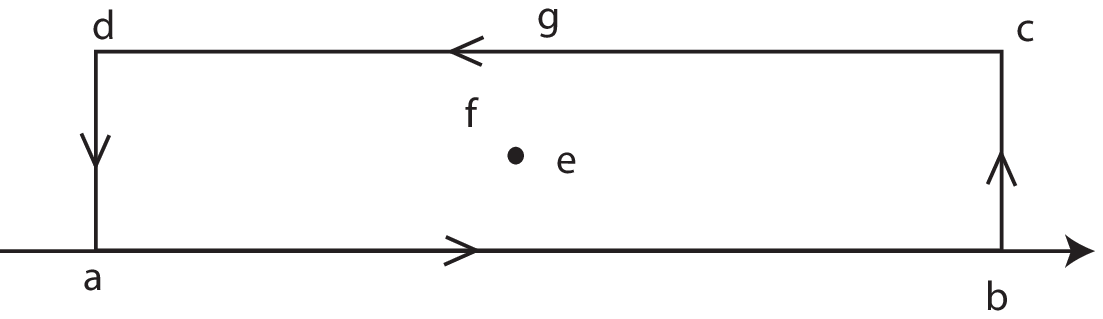}
\end{figure}
We claim this equals the limit of the contour integrals
$$B(1)=\lim_{R\to\infty}\oint_{\Gamma_R} \frac{2g(z)}{e^{z}-e^{-z}}\ dz$$
Where $\Gamma_R$ is the rectangle
$$([-R,R]\times\{0,2\pi i\})\cup (\{\pm R\}\times[0,2\pi i])$$
oriented counterclockwise.
Observe that on the vertical sides $(\{\pm R\}\times[0,2\pi i])$ the integrand goes to $0$ as $R\to\infty$.
The integral along the horizontal sides $([-R,R]\times\{0,2\pi i\})$ with the orientation specified is
$$2\int_{-R}^{R}\frac{g(z)-g(z+2\pi i)}{e^{z}-e^{-z}}\ dz$$
This proves the claim. To evaluate the contour integral we observe the only singularity inside $\Gamma_R$ of
$$\frac{2g(z)}{e^{z}-e^{-z}}$$ is a simple pole
at $z=i\pi$. Now $$2g(i\pi)=2\ \frac{i (i\pi )(-i\pi )}{4\pi}=\frac{i\pi }{2} $$
Also the denominator is $e^z-e^{-z}= 2z +\cdots$ has residue $-1/2$ at $z=i\pi$ because
\begin{align*}
\text{residue} & = \lim_{z\to i \pi}\frac{z-i\pi}{e^z-e^{-z}}\\
& = \lim_{w\to0}\frac{w}{e^{w+i\pi}-e^{w-i\pi}}\\
& =\lim_{w\to 0}\frac{w}{e^{-i\pi}(e^w-e^{-w})}\\
& =-\lim_{w\to 0}\frac{w}{(1+w+\cdots)-(1-w+\cdots)}\\
& =-\lim_{w\to 0}\frac{w}{2w+\cdots}\\
& =-\frac{1}{2}
\end{align*}
 Thus the residue of the integrand at $i\pi$ is $(-1/2)(i\pi/2)=-i \pi/4$. Cauchy's theorem
gives the contour integral is $$(2\pi i)(-i\pi/4)=\pi^2/2$$
\end{proof}

\section{ideal triangulations}

Ideal triangulations of surfaces were introduced by Thurston for hyperbolic surfaces, see \cite{Lack},\cite{Mosher}. 
The extension to properly convex surfaces is
routine.

\begin{definition} 
An {\em ideal triangulation} of a convex projective surface $F=\Omega/\Gamma$ is a decomposition
of $F$ into closed subsets with disjoint interiors called {\em ideal triangles}
such that each component of the preimage of an ideal triangle in $F$ is an ideal triangle in $\Omega$.
\end{definition}

\begin{figure}
	\psfrag{p}{$p$}
	\psfrag{q}{$q$}
	\psfrag{r}{$r$}
	\psfrag{T}{$T$}
	\psfrag{D}{$\Delta$}
	\psfrag{Om}{$\Omega$}
	\includegraphics[scale=0.5]{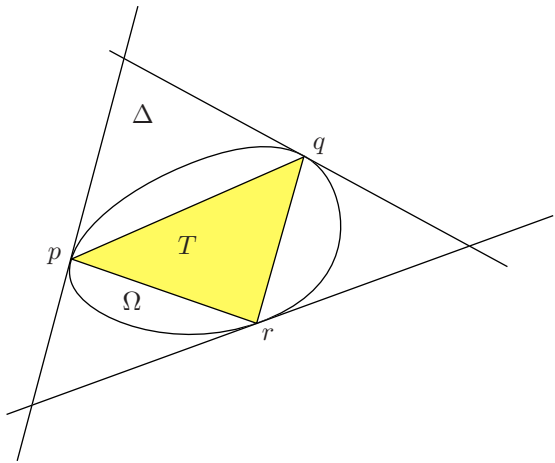}
	\caption{An ideal triangle $T$ in $\Omega$}\label{domaintriangle}
\end{figure}

\begin{proposition}\label{idealtriangnumber} If $F$ is a closed convex projective surface then $F$ has an ideal triangulation.
The number of ideal triangles is $2|\chi(F)|$.
\end{proposition}

Suppose $F=\Omega/\Gamma$ is a compact, strictly convex projective surface with $\chi(F)<0$. Then
$\Omega\subset\RP^2$ has a unique tangent line
at each point of $\partial\overline\Omega$ by  \cite{benz}. If $T\subset\Omega$ is an ideal triangle with vertices $p,q,r\in\partial\Omega$. 
The tangent lines at $p,q,r$ contain the sides of
a triangle $\Delta$ which contains $\Omega$.  Following Fock and Goncharov \cite{FG} we define the {\em
shape} $t=t(T,\Omega)\in[1,\infty)$ of $T$ in $\Omega$ to be the shape $t(T,\Delta)$ of $T$ in $\Delta$ previously defined.

For a properly convex set $\Omega$ we denote the Hilbert area form
on $\Omega$ by $\omega_{\Omega}$. This pushes down to a 2-form
$\omega_F$ on $F$.  We denote the {\em Hilbert area} of a measurable
 subset $X\subset F$ by
$$area_{F}(X)=\int_X\omega_{F}$$ Since $\Omega\subset\Delta$ it follows that $d_{\Delta}\le d_{\Omega}$ 
and $\omega_{\Delta}\le\omega_{\Omega}$. It follows that if $T\subset F$ is an ideal triangle with 
shape parameter $t$
 then
$$area_{F}(T)\ge B(t)\ge\pi^2/2$$
Since $F$ contains $2|\chi(F)|$ ideal triangles with disjoint interiors theorem (\ref{maintheorem}) follows from
(\ref{idealtriangnumber}) and (\ref{triangleprop}).

\small
\bibliography{refs.bib} 
\bibliographystyle{abbrv} 

\gap
\noindent\address{DC: Department of Mathematics, University of California, Santa Barbara, CA 93106}\\
\address{IA: Department of Mathematics and Computer Science, Wesleyan University, Science Tower 655,
265 Church Street, Middletown, CT 06459-0128}
\address{}\\
\email{cooper@math.ucsb.edu}\\
\email{iadeboye@wesleyan.edu}\\

\end{document}